\RequirePackage[l2tabu]{nag} 
\documentclass[11pt]{article}

\usepackage{fullpage}
\usepackage{float}
\usepackage{fancyhdr} 
\usepackage{lastpage} 
\usepackage{extramarks} 
\usepackage{courier} 
\usepackage[top=1in, bottom=1in, left=1in, right=1in]{geometry}

\usepackage[utf8]{inputenc}

\usepackage{amsfonts}
\usepackage[dvipsnames]{xcolor}
\usepackage{amsmath}

\usepackage{cite}

\usepackage{amssymb}

\let\emptyset\varnothing

\usepackage[linesnumbered, ruled, vlined]{algorithm2e}

\SetCommentSty{mycommfont}

\usepackage{graphicx}

\usepackage{enumerate}

\usepackage{bigints}

\makeatletter
\def\endthebibliography{%
	\def\@noitemerr{\@latex@warning{Empty `thebibliography' environment}}%
	\endlist
}
\makeatother

\usepackage{amsthm}

\newtheorem{thm}{Theorem}[section]

\usepackage{amssymb}

\usepackage{authblk}

\makeatletter
\let\NAT@parse\undefined
\makeatother
\usepackage{hyperref}  

\title{
\vspace{-30mm} 
\small{
\normalfont This paper has been accepted for publication at the 23rd IEEE International Conference on Intelligent Transportation Systems, 2020. Please cite the paper as: \\ Yue Guan, Anuradha M. Annaswamy, and H. Eric Tseng. “Towards Dynamic Pricing for Shared Mobility on Demand using Markov Decision Processes and Dynamic Programming.” \textit{2020 IEEE 23rd International Conference on Intelligent Transportation Systems (ITSC).} IEEE, 2020.
}\\
\vspace{10mm} 
\LARGE \bf Towards Dynamic Pricing for Shared Mobility on Demand using Markov Decision Processes and Dynamic Programming
}

\author[1]{Yue Guan\thanks{Corresponding author. Email: guany@mit.edu.}}
\author[1]{Anuradha M. Annaswamy}
\author[2]{H. Eric Tseng}
\affil[1]{Department of Mechanical Engineering, Massachusetts Institute of Technology}
\affil[2]{Research and Advanced Engineering, Ford Motor Company}

\begin{document}
	
\newcommand{\figurepath}{../images}

\maketitle

\begin{abstract}
	
	In a Shared Mobility on Demand Service (SMoDS), dynamic pricing plays an important role in the form of an incentive for empowered passengers to decide on the ride offer. Strategies for determining the dynamic tariffs should be suitably designed so that the incurred demand and supply within the SMoDS platform are balanced and therefore economic efficiency is further achieved. In this manuscript, we formulate a discrete time Markov Decision Process (MDP) to determine the probability of acceptance of each empowered passenger that is desired by the SMoDS platform. The proposed MDP formulation is a versatile framework which is shown to explicitly accommodate passenger behavior and realize the desired system objective. Estimated Waiting Time (EWT) is utilized as a suitable metric to measure the balance between demand and supply, with the goal of regulating EWT around a target value. We propose the use of a Dynamic Programming algorithm to derive the optimal policy that achieves the regulation. Computational experiments are conducted to demonstrate effective regulation of EWT, through various scenarios. 
		
\end{abstract}


\paragraph{Index Terms\textnormal{:}} Shared Mobility on Demand, Dynamic Pricing, Estimated Waiting Time, Markov Decision Process, Dynamic Programming, Lookahead Search, Smart Cities.

%
	

\section{Introduction}
\label{intro}

Shared Mobility on Demand Service (SMoDS) has transformed urban mobility and introduced a continuum of solutions between the traditionally binary modes of private individual vehicles and public mass transit, leading to a range of services with different degrees of cost, flexibility, and carbon footprint. This manuscript pertains to an SMoDS solution that consists of customized dynamic routing and dynamic pricing. We build on our earlier works in \cite{guan2019dynamicrouting, guan2019cpt, annaswamy2018transactive} and propose a Markov Decision Process (MDP) formulation and a Dynamic Programming (DP) algorithm towards dynamic pricing for the SMoDS platform. 

Dynamic pricing has achieved remarkable successes in emerging ride sharing platforms such as Uber, Lyft, and Didi Chuxing, where passengers are empowered to have the option to decide whether to accept or decline the ride offers. Dynamic tariffs provide the incentive signals for the empowered passengers; they directly affect the probability with which the passengers are likely to accept the SMoDS ride offers. Dynamic pricing therefore helps balance the demand and supply within the SMoDS platform through leveraging the probability of acceptance of the empowered passengers, and hence further achieve better economic efficiency. The balance could be measured via several Key Performance Indicators, one of which is the Estimated Waiting Time (EWT). EWT corresponds to the average time that an upcoming passenger will wait until being picked up \cite{cohen2016using}. A large $\text{EWT}(t)$ indicates that demand exceeds supply at time $t$ and vice versa. Our focus in this manuscript is on the regulation of $\text{EWT}(t)$ around a target value $\text{EWT}^*$, by designing the probability of acceptance for each empowered passenger that is desired by the SMoDS platform, through the proposed MDP formulation and DP algorithm.  

We assemble two integral components, an MDP formulation that serves as the underlying framework to determine the desired probability of acceptance $p^*$ for each empowered passenger, and a DP algorithm to derive the optimal policy of the MDP for a specific scenario with an offline setup. The two components will help integrate human behavior into the SMoDS framework in an efficient manner. By combining $p^*$ determined via the methodology proposed in this manuscript with the passenger behavioral model developed in \cite{guan2019cpt}, the dynamic tariff that nudges the passenger towards $p^*$ and hence enhances the balance between demand and supply can be designed for the SMoDS platform. We demonstrate using computational experiments that $\text{EWT}(t)$ can be effectively regulated around $\text{EWT}^*$ for various $\text{EWT}^*$ values and for a time-varying $\text{EWT}^*$ as well. The extension of the DP algorithm to online scenarios is also discussed.

Dynamic pricing for the SMoDS has been quite a popular research topic during recently years \cite{lu2018surge, hu2019surge, ma2018spatio, banerjee2015pricing, korolko2018dynamic, kanoria2019near}. 
\cite{banerjee2015pricing} develops a queueing-theoretic economic model for dynamic pricing and proves that it is more robust
than static pricing.
\cite{kanoria2019near} develops a near optimal control for ride hailing platforms via mirror backpressure that maximizes payoff and derives the regret bound. 
\cite{ma2018spatio} proposes the spatio-temporal pricing mechanism that has prices be smooth in space and time therefore drivers will not decline the dispatched rides to seek ones with higher returns nearby. 
The main distinction between these earlier publications and our contributions here in this manuscript is the explicit accommodation of behavioral modeling for the empowered passengers in the SMoDS design through capturing their probability of acceptance of the service. In contrast to these earlier studies which require various conditions including discretization of locations and stationary states, our proposed problem formulation requires minimal assumptions and can be readily extended to a broad range of scenarios with varying demand patterns and operational objectives.

\section{Preliminaries: Dynamic Routing and Dynamic Pricing}
\label{background}

The SMoDS solution we are developing consists of dynamic routing and dynamic pricing, which delivers a customized dynamic route and dynamic tariff to each passenger \cite{guan2019dynamicrouting, annaswamy2018transactive, guan2019cpt}. The overall schematic is illustrated in Fig. \ref{fig:schematic}, with three building blocks functioning in the following manner. When receiving a new ride request, the first block derives an optimized dynamic route that is to be offered to the passenger. The second and third blocks ensure dynamic pricing that accommodates the possibility that the empowered passenger may either accept or reject the ride offer and still generates an overall performance desired by the SMoDS platform. Of these, the second concerns a passenger behavioral model that derives the actual probability of acceptance for a specified dynamic tariff. The third and final block is the construction of a desired probability of acceptance for the passenger that will ensure the desired performances by the SMODS platform. Using this third block, one can then design the dynamic tariff that nudges the passenger towards the desired probability of acceptance by solving the inverse problem. We focus on determining desired probability of acceptance in this manuscript while only briefly introduce dynamic routing and passenger behavioral modeling in this section, and refer the readers to \cite{guan2019dynamicrouting} and \cite{guan2019cpt} for mores details. 
\begin{figure}[h!]
	\centering
	\includegraphics[width=0.95\textwidth]{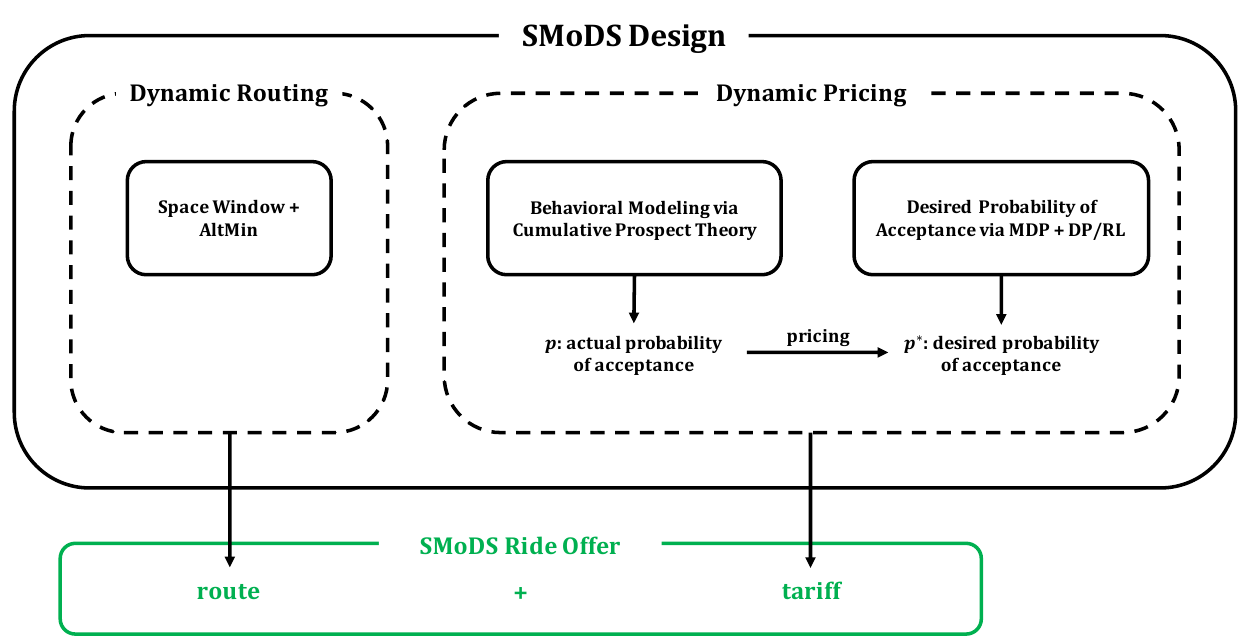}
	\caption{Overall schematic of the SMoDS design.}
	\label{fig:schematic}
\end{figure}

\subsection{Dynamic Routing}
\label{dynamicrouting}

An Alternating Minimization (AltMin) based optimization algorithm is developed for dynamic routing with added spatial flexibility enabled by space window \cite{guan2019dynamicrouting}. AltMin derives dynamic routes for passengers given their requested pickup, drop-off locations, willingness to walk and other service requirements that minimize a weighted sum of various travel time cost terms. AltMin has been demonstrated via various computational experiments to outperform classic constrained optimization formulations that are solved by standard solvers in terms of both computational complexity and optimality. 

In the current SMoDS design, AltMin on the one hand derives dynamic routes that are provided in the ride offer, on the other hand contributes to the MDP formulation towards dynamic pricing, by means of the functions $F_{\text{DR}}(\cdot)$, $F_{\text{wait}}(\cdot, \cdot)$, and $F_{\text{EWT}}(\cdot)$ which we formally introduce in Section \ref{MDP}.   

\subsection{Dynamic Pricing}
\label{dynamicpricing}

Dynamic pricing is divided into two phases as follows.  

\subsubsection{Passenger Behavioral Modeling}
\label{behavioral}

The first phase is passenger behavioral modeling. A passenger behavioral model is one that inputs the specifications of the SMoDS ride offer and alternative transportation options, then outputs the actual probability with which the passenger takes the SMoDS. According to discrete choice model \cite{ben1985discrete}, the actual probability of accepting the $\ell^{\text{th}}$ transportation option given $Q \in \mathbb{Z}_{>0}$ options to choose from is given by 
\begin{equation}
	\label{discrete_choice_model}
	p^{\ell} = \frac{e^{U_{\ell}}}{\sum_{q=1}^Q e^{U_q}}, \quad \forall \ell \in \{1, \cdots, Q\}
\end{equation}
Here we apply Cumulative Prospect Theory (CPT) in passenger behavioral modeling to capture subjective decision making of passengers when facing risk or uncertainty. This is because the SMoDS is exposed to uncertainty since the vehicles need to accommodate new passengers at anytime during the route therefore the service quality is to some extent stochastic. $U_{\ell}$ denotes the utility of taking the $\ell^{\text{th}}$ option that is subjectively perceived by the passenger, which can be computed via the framework we developed in \cite{guan2019cpt}.

\subsubsection{Desired Probability of Acceptance}
\label{desired}

Given the ride specifications, the passenger is empowered to take the SMoDS with certain probability. The probability of acceptance impacts the expected performances of the SMoDS platform. With the dynamic routes derived via AltMin, dynamic pricing serves as an incentive to tune this probability. As discussed in Section \ref{intro}, the goal is to regulate $\text{EWT}(t)$ around the target value $\text{EWT}^*$ and hence to enhance the balance between demand and supply within the SMoDS platform. Therefore, the probability of acceptance desired by the SMoDS platform should minimize the average deviation of the incurred $\text{EWT}(t)$ from $\text{EWT}^*$, for which we define
\begin{equation}
	\label{objective}
	R_{\text{total}} = - \frac{1}{T}\int_{0}^{T}|\text{EWT}(t) - \text{EWT}^*| \; dt	
\end{equation}
where $T > 0$ denotes the evaluation horizon. 
The problem can thus be posed as the maximization of the reward $R_{\text{total}}$ in (\ref{objective}), the negative of the average deviation of the incurred $\text{EWT}(t)$ from $\text{EWT}^*$.
Note that how to choose an appropriate $\text{EWT}^*$ value is beyond the scope of this manuscript, which requires knowledge of the demand and supply of the SMoDS platform, i.e., request pattern and fleet portfolio, and explicitly stated objectives, e.g., weighted combination of revenue and ridership. This is one of our future directions. We assume that $\text{EWT}^*$ is given unless otherwise stated. 

From (\ref{discrete_choice_model}), we obtain $p = f(\mu)$ the actual probability of acceptance from each empowered passenger as a function of the dynamic tariff $\mu$. 
The desired value $p^*$ is to be derived via the methodology developed in this manuscript. We then derive the dynamic tariff ${\mu}^*$ that nudges the passenger towards $p^*$ by simply setting $p = p^*$ and solving the inverse problem ${\mu}^* = f^{-1}(p^*)$.     

\section{A Markov Decision Process Formulation}
\label{MDP}

In this section, we propose a discrete time MDP formulation for determining the optimal desired probability of acceptance for each passenger that makes a ride request.
The agent is the SMoDS server, while the environment denotes everything else including both the fleet and passengers. 
The MDP is defined using the tuple $\langle \mathcal{S}, \mathcal{A}, \mathcal{P}, \mathcal{R}, \gamma \rangle$ \cite{sutton2018reinforcement}, which denote the state space, action space, state transition function, reward function, and discount factor, respectively. 
The reward function corresponds to the regulation of $\text{EWT}(t)$ around the target value $\text{EWT}^*$.
The objective is to derive the optimal policy for the agent so that the corresponding action is the optimal desired probability of acceptance for each new passenger. This optimal desired probability of acceptance in turn will maximize the expected total discounted rewards.  
In what follows we will describe each component of the MDP formulation in greater detail. 

We first make a few comments regarding the formulation: 
\begin{itemize}
	\item The overall problem is hybrid in nature, with ride requests occurring hence actions taken at $k \in \mathbb{Z}_{>0}$, which correspond to discrete time instances $t_r = k \Delta t_r$ with $\Delta t_r \in \mathbb{R}_{>0}$ defined as the time interval between any two consecutive requests. The state of the environment is defined in continuous time as it reflects continuous quantities such as travel times.  
	\item The notations followed are similar to those in Sutton \cite{sutton2018reinforcement} and Silver \cite{silver2015reinforcement}. Whenever an uppercase symbol is used to denote a random variable, the corresponding lowercase symbol is assumed to denote a realization of that random variable. Calligraphic symbols are used to represent the state and action spaces as well as the state transition and reward functions. 
\end{itemize}

\subsection{State Space $\mathcal{S}$}
\label{state}

We denote $S_t \in \mathcal{S}$ as the state of the environment at time $t \in \mathbb{R}_{\geq 0}$, which corresponds to 
both the fleet and passengers as well as any unprocessed ride request. Denote $m_t, n_t \in \mathbb{Z}_{\geq 0}$ as the numbers of active vehicles and passengers at $t$ within the SMoDS platform, respectively. Then
\begin{equation}
	\label{define_S_t}
	S_t = \Big[\big\{S_t^{v_i}\big\}_{i \in [m_t]}, \big\{S_t^{p_j}\big\}_{j \in [n_t]}, \Omega_{t} \Big]
\end{equation}  
where $S_t^{v_i}$ and $S_t^{p_j}$ denote the state of vehicle $v_i$ and passenger $p_j$, respectively. $[X] \triangleq \{1,2,\cdots,X\}$, $\forall X \in \mathbb{Z}_{\geq 0}$. $\Omega_t$ denotes the specifications of the ride request received by the agent at $t$ with distribution $f_{\Omega_{t}}(\omega_{t})$ which is assumed to be known and independent on the other components of the state. $\Omega_t = \emptyset$ if $t \neq t_r$. 

$S_t^{v_i}$ and $S_t^{p_j}$ are defined in (\ref{state_vehicle}) and (\ref{state_passenger}), respectively. 	
\begin{equation}
	\label{state_vehicle}
	S_t^{v_i} = [L_t^{v_i}, G_t^{v_i}, O_t^{v_i}]
\end{equation}
$L_t^{v_i}$ denotes the location of vehicle $v_i$, $G_t^{v_i}$ denotes a set of consecutive dynamic routing points of the route of $v_i$, and $O_t^{v_i}$ corresponds to other status indicators of $v_i$, all at time $t$. Examples of $O_t^{v_i}$ include the number of passengers onboard. 
\begin{equation}
	\label{state_passenger}
	S_t^{p_j} = [L_t^{p_j}, G_t^{p_j}, O_t^{p_j}]
\end{equation}  
$L_t^{p_j}$ indicates the location of passenger $p_j$, $G_t^{p_j}$ denotes the requested and negotiated pickup and drop-off locations therefore the dynamic route of $p_j$, and similarly $O_t^{p_j}$ corresponds to other status indicators of $p_j$, all at time $t$. Examples of $O_t^{p_j}$ include the positions of $p_j$ in the request, pickup, and drop-off queues if the maximum position shift constraints are imposed \cite{guan2019dynamicrouting}. 

\subsection{Action Space $\mathcal{A}$}
\label{action}

We denote $A_{t_r} \in \mathcal{A}$ as the action taken by the agent at $t_r$. The action space of $A_{t_r} = a_{t_r}$ is given by
\begin{equation}
	\label{action_range}
	a_{t_r} \in \Big [\underline{a_{t_r}}, \overline{a_{t_r}} \Big ], \, 0 < \underline{a_{t_r}} < \overline{a_{t_r}} < 1
\end{equation}  
where $\underline{a_{t_r}}$ and $\overline{a_{t_r}}$ denote the lower and upper bounds of $a_{t_r}$ respectively, both of which are functions of the ride specifications of the SMoDS and the alternative transportation options. 
Since $a_{t_r}$ corresponds to the desired probability of acceptance and the dynamic tariff in turn is affected by $a_{t_r}$ which should be charged within a reasonable range, we choose tighter bounds as in (\ref{action_range}).
The optimal action that maximizes the expected total discounted rewards is stated in Theorem \ref{thm:optimal_action}.  
\begin{thm}
	\label{thm:optimal_action}
	$\forall t_r$, we have the optimal action
	\begin{equation}
		\label{optimal_action}
		a^*_{t_r}  \in \Big \{ \underline {a_{t_r}}, \overline{a_{t_r}} \Big \}
	\end{equation} 
\end{thm}
The proof is provided in the Appendix.

Once the action is taken by the agent and delivered to the passenger, the passenger responds with the decision on the ride offer. We denote $D_{t_r}$ as the passenger decision at $t_r$, where $d_{t_r} = 1$ represents acceptance while $d_{t_r} = 0$ represents rejection. $D_{t_r}$ obeys the Bernoulli distribution as
\begin{equation}
	\label{d_distribution}
	D_{t_r} \sim \text{B}(1, a_{t_r})
\end{equation} 

\subsection{State Transition Function $\mathcal{P}$ }
\label{transition}

In this subsection, we derive the state transition function $\mathcal{P}_{xy}^z$ from state $x$ to state $y$ under action $z$. If no action $z$ is taken, the superscript is omitted.
Three different scenarios are presented, where the first is due to continuous time dynamics while the others are due to events occurred at discrete instances.
For the latter, recall that $t_r$ denotes the time instance when a new ride request is received, and we define $t_r^-$ and $t_r^+$ as time instances right before and after the passenger responds to the agent action with a decision.
The three scenarios to evaluate the state transition are detailed as follows.   

\subsubsection{State Transition due to Internal Dynamics}
\label{transition_one}
When there is no new ride request received nor existing request processed, the transition is purely due to the internal dynamics of the environment, i.e., movements of vehicles and passengers following the routes, and any incurred status changes. This type of state transition is deterministic and described as
\begin{equation}
	\label{transition_no}
	\mathcal{P}_{s_t s_{t + \tau}} = 
		\begin{cases}
			1, \, & \text{if} \, s_{t + \tau} = F_{\text{ID}}(s_t, \tau) \\
			0, \, & \text{otherwise}
		\end{cases}
\end{equation} 
(\ref{transition_no}) is valid when no request is received or processed within $[t, t + \tau], \forall t, \tau \geq 0$. $F_{\text{ID}}(\cdot, \cdot)$ captures the internal dynamics.

\subsubsection{State Transition due to Receiving a New Request}
\label{transition_two}
When there is a new request $\omega_{t_r}$ received at $t_r$, the state transition immediately after is given by
\begin{equation}
	\label{transition_yes_receive}
	\mathcal{P}_{s_{t_r^-} s_{t_r}} = 
	\begin{cases}
	f(\omega_{t_r}), \, & \text{if} \, s_{t_r} =[s_{t_r^-} , \omega_{t_r}] \\
	0, \, & \text{otherwise} 
	\end{cases}
\end{equation}  
$f(\omega_{t_r})$ denotes the distribution of $\Omega_{t_r}$ where the subscript $\Omega_{t_r}$ is omitted for ease of notation. 
$s_{{t_r}^-}$ and $s_{{t_r}^+}$ denote the states right before and after the server processes the request received at $t_r$ and the passenger decides on the ride offer, respectively.

\subsubsection{State Transition due to Processing an Existing Request}
\label{transition_three}
When the agent takes an action to process the request and the passenger responds with the decision on the offer, a state transition occurs by updating the status of the vehicles and passengers according to the decision and eliminating the corresponding request if rejected. That is, the state transition immediately after a request-processing is given by 
\begin{equation}
	\label{transition_yes}
	\mathcal{P}_{s_{t_r} s_{t_r^+}} ^ {a_{t_r}} = 
		\begin{cases}
			a_{t_r}, \, & \text{if} \, s_{t_r^+} = F_{\text{DR}}(s_{t_r}) \\
			1 - a_{t_r}, \, & \text{if} \, s_{t_r^+} = s_{t_r^-} \\
			0, \, & \text{otherwise} 
		\end{cases}
\end{equation}  
(\ref{transition_yes}) states that the passenger accepts the SMoDS ride offer with the probability of $a_{t_r}$ and as a result the state $s_{t_r}$ transits to $s_{t_r^+} = F_{\text{DR}}(s_{t_r})$. $F_{\text{DR}}(\cdot)$ denotes the dynamic routing algorithm, 
which updates the routes and hence the state transits due to accommodating the new passenger, 
and is a function of $s_{t_r}$. The passenger may also decline the offer with the probability of $(1 - a_{t_r})$, in which case, 
$s_{t_r^+} = s_{t_r^-}$. 
 
The three scenarios discussed above have three distinct effects on the state transition. That is, either the states evolve only due to the internal dynamics $F_{\text{ID}}(\cdot, \cdot)$, due to the receiving of a new request, or due to the processing of an existing request. We now combine (\ref{transition_no}) through (\ref{transition_yes}) and describe the complete state transition over the interval $[k \Delta t_r, (k+1) \Delta t_r]$, which includes one processing of an existing request at $k \Delta t_r$ and one receiving of a new request at $(k+1)\Delta t_r$ as follows:
\begin{equation}
\label{transition_combine}
	\mathcal{P}_{s_k s_{k+1}} ^ {a_k} = 
		\begin{cases}
			a_kf(\omega_{k+1}), \, & \text{if} \,  s_{k+1} = [F_{\text{ID}} ( F_{\text{DR}}(s_k), \Delta t_r ), \omega_{k+1} ] \\
			(1 - a_k)f(\omega_{k+1}), \, & \text{if} \, s_{k+1} = [F_{\text{ID}}(s_{k^-}, \Delta t_r), \omega_{k+1}] \\
			0, \, & \text{otherwise} 
		\end{cases}
\end{equation}   
The three scenarios and the overall discrete time state transition from $s_k$ to $s_{k+1}$ are illustrated in Fig. \ref{fig:transition}.
\begin{figure}[h!]
	\centering
	\includegraphics[width=0.95\textwidth]{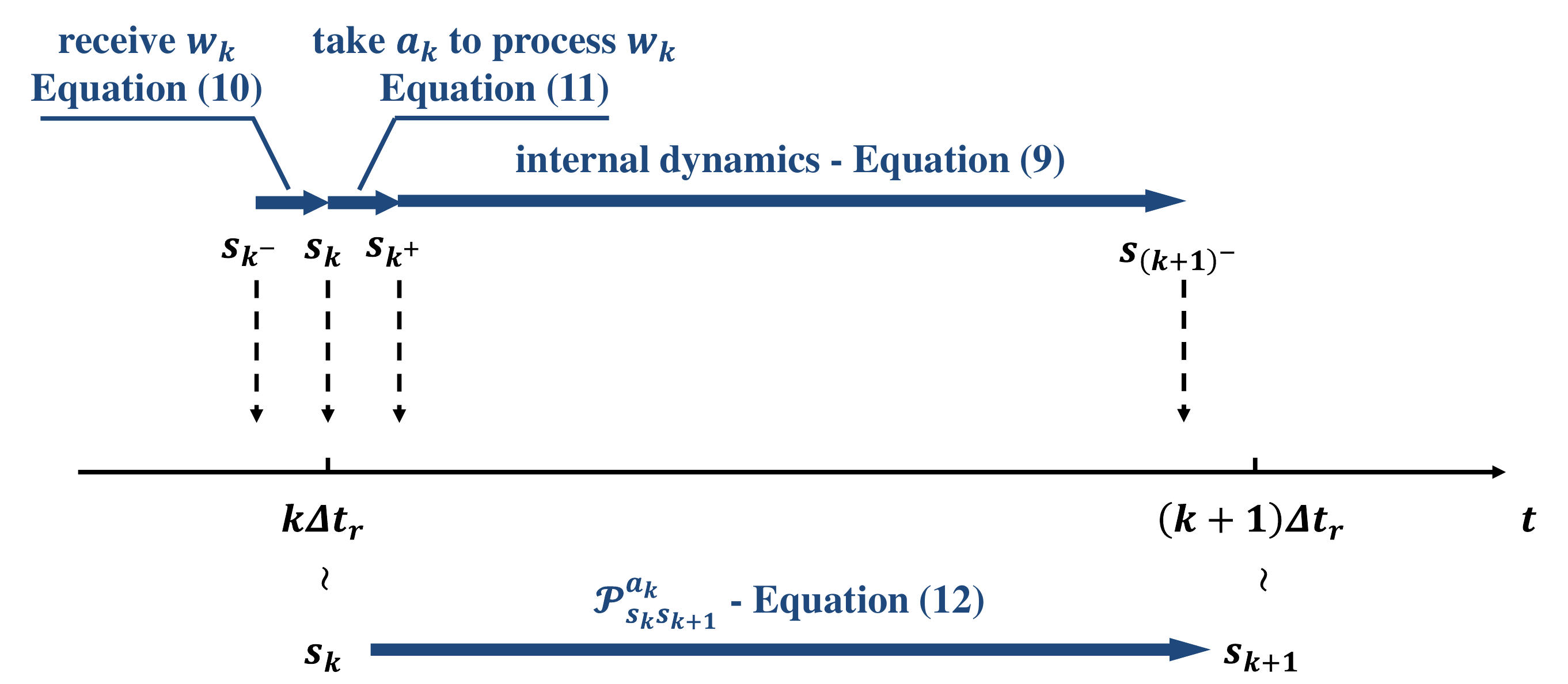}
	\caption{Illustration of state transition.}
	\label{fig:transition}
\end{figure}

\begin{thm}
	\label{thm:markov}
	The discrete time states $S_{t_r}$ defined in (\ref{define_S_t}) through (\ref{state_passenger}) with the transition function derived in (\ref{transition_combine}) are Markov.
\end{thm}

The proof is provided in Appendix.

\subsection{Reward Function $\mathcal{R}$}
\label{reward}

In the following $R_x^y$ denotes the reward function at state $x$ following an action $y$, and $\mathbb{E}_X$ denotes the expectation of a random variable $X$. The total reward defined in (\ref{objective}) is examined over one interval $[t_r^+, {(t_r + \Delta t_r)}^-]$, i.e., immediately after the agent takes the action at $t_r$ and right before the next action is about to take place at $(t_r + \Delta t_r)$. This is given by 
\begin{equation}
	\label{reward_function}
		\begin{split}
			\mathcal{R}_{s_{t_r}}^{a_{t_r}} = - \mathbb{E}_{D_{t_r}} \bigg[ \frac{1}{\Delta t_r} \int_{t_r^+}^{{(t_r + \Delta t_r)}^-}  |\text{EWT}(\tau) - \text{EWT}^*| \; d\tau \bigg] 
		\end{split}
\end{equation}
It is clear from (\ref{reward_function}) that the quantity on the right-hand side represents the immediate impact of action $a_{t_r}$ on the regulation of $\text{EWT}(t)$ around $\text{EWT}^*$, before the next action. 

In order to compute this reward function, we note that $\text{EWT}(t)$ is a function of the state $s_t$ and is of the form 
\begin{equation}
	\label{calculate_EWT}
	\text{EWT}(t) = \mathbb{E}_{\Omega_{t}} [ F_{\text{wait}}(s_t , \omega_{t}) ] = F_{\text{EWT}}(s_t)
\end{equation} 
where $F_{\text{wait}}(\cdot , \cdot)$ denotes the waiting time derived via the dynamic routing algorithm as in \cite{guan2019dynamicrouting} for a given $s_t$ and $\omega_{t}$.

\subsection{Discount Factor $\gamma$}
\label{discount}

The discount factor $\gamma \in [0, 1]$. Typically $\gamma < 1$, while if the episode is guaranteed to terminate, we could let $\gamma  = 1$.

\subsection{Value Function}
\label{value}

With the discrete time MDP formulation elaborated in Sections \ref{state} through \ref{discount}, we rewrite the objective function in (\ref{objective}) using the reward function in (\ref{reward_function}) as the expected total discounted rewards in the form of 
\begin{equation}
	\label{objective_rewrite}
	R_{\text{total}} = \mathbb{E} \Bigg[ \sum_{k = 0} ^ {\infty} \gamma^k \mathcal{R}_{S_{k \Delta t_r}} ^ {A_{k \Delta t_r}} \Bigg]
\end{equation}
where we let $A_0 = \emptyset$ for ease of notation. The optimal policy $\pi^*(\cdot)$ that maximizes (\ref{objective_rewrite}) can now be derived. 
Using a value based approach, the optimal value function of any state $S_{t_r} = s_{t_r}$ is defined as
\begin{equation}
	\label{value_state}
	v^*(s_{t_r}) = \max_{\pi(\cdot)} \mathbb{E} \Bigg[ \sum_{k = 0} ^ {\infty} \gamma^k \mathcal{R}_{s_{t_r + k \Delta t_r}} ^ {\pi(s_{t_r + k \Delta t_r})} \; \Bigg  | \; s_{t_r}\Bigg]
\end{equation} 
where the expectation is taken over the upcoming requests and the decisions (including $D_{t_r}$) from the passengers in response to the corresponding actions. Using Bellman Optimality Equation, we rewrite (\ref{value_state}) as 
\begin{equation}
	\label{value_state_rewite}
	v^*(s_{t_r}) = \max_{a_{t_r}} \Big\{\mathcal{R}_{s_{t_r}}^{a_{t_r}} + \gamma  \mathbb{E} \Big[v^*(s_{t_r + \Delta t_r}) \; \Big| \; s_{t_r}\Big] \Big\}
\end{equation} 
where the expectation is taken over $D_{t_r}$ and $\Omega_{t_r + \Delta t_r}$. Using (\ref{value_state_rewite}), the optimal policy can be described as 
\begin{equation}
	\label{optimal_policy}
	{\pi}^*(s_{t_r}) = \text{arg}\max_{a_{t_r}} \Big\{ \mathcal{R}_{s_{t_r}}^{a_{t_r}} + \gamma \mathbb{E} \Big[ v^*(s_{t_r + \Delta t_r}) \; \Big| \; s_{t_r} \Big] \Big\}
\end{equation}

When the numbers of active vehicles and passengers are large, and considering that the state space is continuous, deriving $\pi^*(\cdot)$ through $v^*(\cdot)$ exactly may be computationally impractical. Alternatively, one could apply function approximation \cite{bertsekas1995neuro} to derive a parametrized suboptimal policy as 
\begin{equation}
	\label{value_state_rewite_approximation}
	\begin{split}
		\tilde{\pi}^*(s_{t_r};f_v) = \text{arg}\max_{\tilde{a}_{t_r}} \Big\{ \mathcal{R}_{s_{t_r}}^{\tilde{a}_{t_r}} + \gamma \mathbb{E} \Big[ \tilde{v}^*(s_{t_r + \Delta t_r};f_v) \; \Big| \; s_{t_r} \Big] \Big\}
	\end{split}
\end{equation}
where $\tilde{v}^*(\cdot;f_v)$ is an approximation of $v^*(\cdot)$ parametrized by the feature vector $f_v$, whose dimension is much smaller than that of the state and therefore results in more efficient computation in a broad range of applications. Examples of $f_v$ in this context could include occupancy rate of the fleet, number of passengers waiting in the pickup queue, etc. We do not address the details of such function approximations in this manuscript, and is relegated to future work.

With these, the MDP formulation $\langle \mathcal{S}, \mathcal{A}, \mathcal{P}, \mathcal{R}, \gamma \rangle$ is detailed in Equations (\ref{define_S_t}) through (\ref{calculate_EWT}) and we will elaborate the Dynamic Programming algorithm based on (\ref{optimal_policy}) that solves a special offline case of the MDP in Section \ref{algorithms}. 

\section{Dynamic Programming for Determining $\pi^*(\cdot)$}
\label{algorithms}

Through out the rest of the manuscript, we focus on a special case of the MDP which adopts an offline setup. A DP algorithm that solves this MDP is developed and outlined in this section. The offline case demonstrates a proof of concept that $\text{EWT}(t)$ could be effectively regulated around $\text{EWT}^*$ through leveraging the probability of acceptance of empowered passengers. The DP algorithm is readily extended for the general online case as well.

With the offline setup, it is assumed that all ride requests are deterministic and known to the SMoDS server a priori. These requests are processed sequentially, meaning that the ride offers are distributed to the passengers sequentially. The action for the current passenger takes into account future ones and the action for the upcoming passenger may adapt according to the decisions from previous ones. This reduces the cardinality of the state space from continuum to a finite one. In addition, according to Theorem \ref{thm:optimal_action}, the space of each $A_{t_r}$ is finite and has a cardinality of $2$. These enable the development of the value based DP algorithm that derives $\pi^*(\cdot)$ exactly to be feasible. Since the setup is offline, the episodes are guaranteed to terminate and therefore we set $\gamma = 1$.

\begin{figure}[h!]
	\centering
	\includegraphics[width=0.95\textwidth]{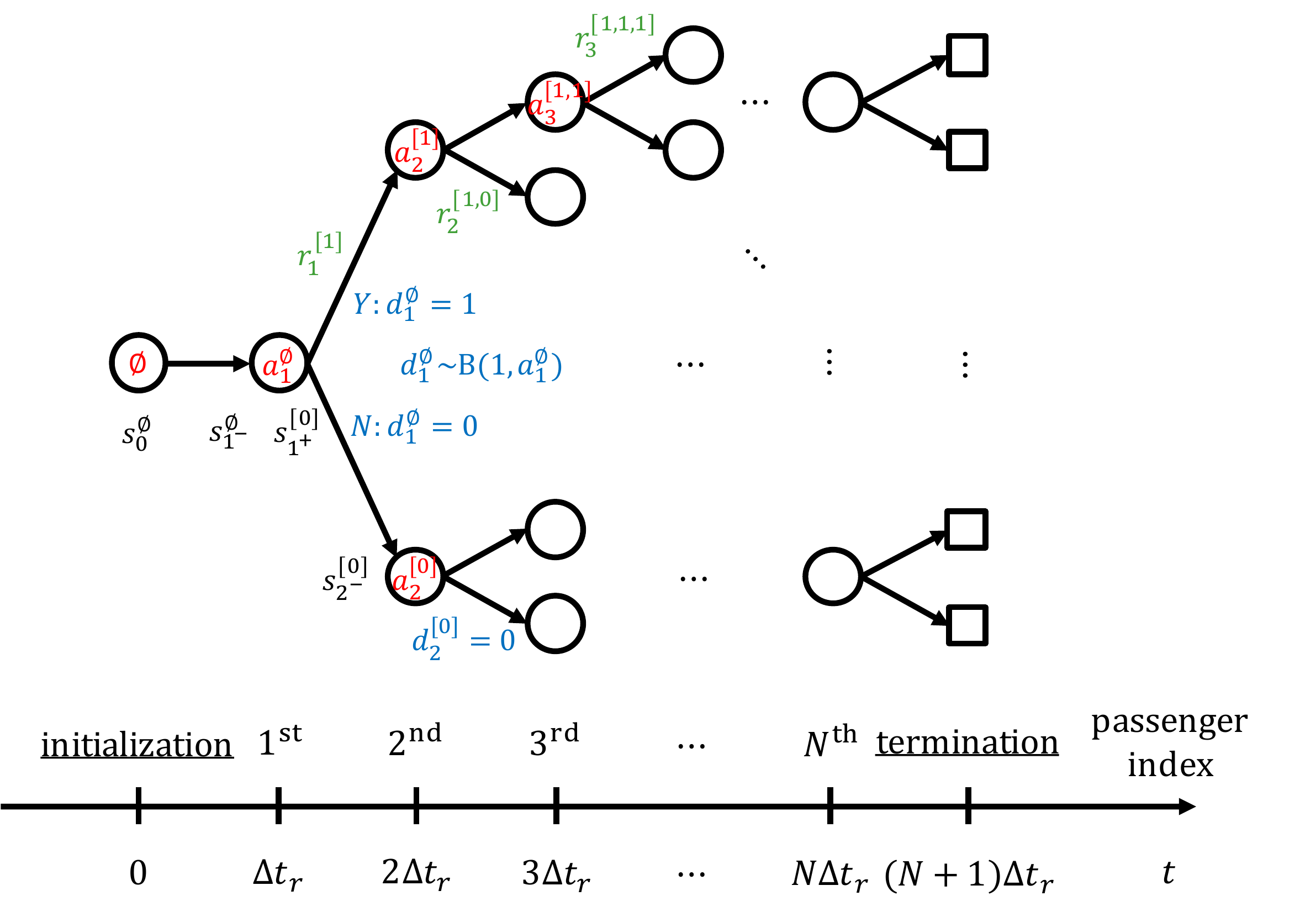}
	\caption{Illustration of the DP algorithm to derive $\pi^*(\cdot)$.}
	\label{fig:tree}
\end{figure}

We now set up the DP algorithm as illustrated in Fig. \ref{fig:tree}. Each vertex in the directed and rooted tree structure shown in the figure represents a possible state where the agent takes an action, and has in turn two children depending on the decision made by the passenger in response. An acceptance of the ride offer by the passenger corresponds to the top branch and a rejection to the bottom, with the action itself indicated in each vertex. The action at the root is denoted as null as no action is taken therein, and no action is indicated at the terminal vertices, represented by black squares. 
Each edge represents the transition after the decision of the previous passenger until a new request is received and a subsequent action is to be taken, and is associated with the realization of the reward. The height of the tree is $H = N + 1$ and each path with length $H$ represents a complete episode as a result of consecutive actions of the agent and decisions from each passenger. $d_{[k]} \triangleq [d_1, \cdots, d_k], \forall k \in [N]$ represents the decisions of the first $k$ passengers. We have $d_{[k]}$ in the superscripts explicitly since the decisions of previous passengers determine which vertex and edge that the DP algorithm traverses. Moreover, we 
let $\omega_{N+1} = \emptyset$ for ease of notation.

The algorithm consists of two steps. The first step is to conduct dynamic routing to process ride requests consecutively which essentially traverses the tree from the root to all leaves, therefore to derive the reward using functions $F_\text{ID}(\cdot, \cdot), F_\text{DR}(\cdot), F_\text{EWT}(\cdot)$ and (\ref{reward_function}) on each edge. The second step is to conduct backward recursion from the leaves to the root, therefore to derive the optimal value function at each state and the corresponding optimal action at each vertex, via Bellman Optimality Equation delineated in (\ref{optimal_policy}). The details of the exact DP algorithm are outlined in Algorithm \ref{DP_exact}.

\begin{algorithm}
	\SetAlgoLined
	\DontPrintSemicolon
	Initialize the route via AltMin and derive $s_{0^+}^{\emptyset}$\;
	Initialize $s_1^{\emptyset} \leftarrow [F_{\text{ID}}(s_{0 ^ +}, \Delta t_r),\omega_1], d_{[0]} \leftarrow \emptyset$\;
	Initialize $v^* \Big( s_{k}^{d_{[k-1]}} \Big) \leftarrow 0, \forall k \in [N + 1], d_{[k-1]}$\; \label{algo_DPE_start}
	\For(\tcp*[f]{step 1: traverse}){$k = 1:N$}{ \label{algo_DPE_mark}
		\For{$d_{[k-1]}$}{
			\For{$d_k = 0:1$}{
				$d_{[k]} \leftarrow \big[ d_{[k-1]}, d_k \big]$\;
				\uIf{$d_k=0$}{$s_{k^+}^{d_{[k]}} \leftarrow s_{k^-}^{d_{[k-1]}}$}
				\Else{$s_{k^+}^{d_{[k]}} \leftarrow F_{\text{DR}} \Big( s_{k}^{d_{[k-1]}} \Big) $}
				$r_k^ {{d_{[k]}}} \leftarrow - \int_{0^+}^{1^-}  \Big|F_{\text{EWT}} \Big(s_{k+\tau}^{d_{[k]}}\Big) - \text{EWT}^* \Big| \; d\tau $\;
				$s_{{(k+1)}^-}^{d_{[k]}} \leftarrow F_{\text{ID}} \Big(s_{k^+}^{d_{[k]}}, \Delta t_r \Big)$\;
				$s_{{k+1}}^{d_{[k]}} \leftarrow \Big[s_{{(k+1)}^-}^{d_{[k]}}, \omega_{k+1} \Big]$\;
			}
		}		
		
	}
	\For(\tcp*[f]{step 2: backup}){$k = N:-1:1$}{
		\For{$d_{[k-1]}$}{
			\uIf{$r_k^ {\big[d_{[k-1]}, 1\big]} + v^*\Big(s_{{k+1}}^{[d_{[k-1]}, 1]} \Big) \geq r_k^ {\big[d_{[k-1]}, 0\big]} + v^* \Big(s_{{k+1}}^{[d_{[k-1]}, 0]} \Big)$ \label{algo_compare_start}}{
				$a_k^{d_{[k-1]}} \leftarrow \overline{a_k}$ 
			}
			\Else{
				$a_k^{d_{[k-1]}} \leftarrow \underline{a_k}$
			}\label{algo_compare_end}
			$v^* \Big(s_{k}^{d_{[k-1]}} \Big) \leftarrow a_k^{d_{[k-1]}} \bigg[r_k^ { \big[d_{[k-1]}, 1 \big]} + v^* \Big(s_{{k+1}}^{[d_{[k-1]}, 1]} \Big) \bigg] + \Big(1 -  a_k^{d_{[k-1]}}\Big) \bigg[r_k^ {\big[d_{[k-1]}, 0\big]} + v^*\Big(s_{{k+1}}^{[d_{[k-1]}, 0]}\Big)\bigg]$
		}
	}
	\label{algo_DPE_end} 
	\caption{Exact-DP Algorithm - $\text{E-DP}(N)$ \label{DP_exact}}
\end{algorithm}

Similar idea can be exploited to develop a heuristic DP algorithm that improves computational efficiency at the cost of optimality, which is summarized in Algorithm \ref{DP_heuristics}. Instead of conducting forward search till the end of the episode, i.e., number of lookahead steps being $N$, Algorithm \ref{DP_heuristics} conducts forward search with steps $\tilde{N} \leq N$ , which serves as a hyper parameter to tune the trade-off between efficiency and optimality. For example, to derive $a_k^{d_{[k-1]}}$, Algorithm \ref{DP_heuristics} essentially carries out E-DP($\tilde{N}$) on the sub tree rooted at vertex $a_k^{d_{[k-1]}}$ with a height of $\tilde{N}$, instead of traversing the entire tree in Fig. \ref{fig:tree}. 

\begin{algorithm}
	\SetAlgoLined
	\DontPrintSemicolon
	\For{$p = 1 : N - \tilde{N} + 1$}{
		\For{$d_{[p-1]}$}{
			$s_{1}^{\emptyset} \leftarrow s_{p}^{d_{[p-1]}}$\;
			\uIf{$1 \leq p < N - \tilde{N} + 1$}{
				execute lines \ref{algo_DPE_start}-\ref{algo_DPE_end} in Algorithm \ref{DP_exact} and derive $a_p^{d_{[p-1]}}$ \label{algo_DPE_mark_mark}
			}
			\Else{
				execute lines \ref{algo_DPE_start}-\ref{algo_DPE_end} in Algorithm \ref{DP_exact} and derive $a_q^{d_{[q-1]}}, \forall N - \tilde{N} + 1 \leq q \leq N$
			}
		}
	} 
	\caption{Heuristic-DP Algorithm - $\text{H-DP}(\tilde{N})$ \label{DP_heuristics}}
\end{algorithm}

\section{A Numerical Case Study}
\label{numerical}

In this section, we present a numerical evaluation of the MDP formulated in Section \ref{MDP} solved via Algorithms \ref{DP_exact} and \ref{DP_heuristics} using a specific case study. Extensions to a general online setup are discussed at the end of this section. 

\subsection{Problem Setup}
\label{setup}

We consider an SMoDS which consists of a single vehicle of capacity 6. 12 ride requests are assumed to be generated, with each request corresponding to a single passenger, and are to be served using this vehicle. The origins and destinations are generated uniformly in a square of one by one mile. The first 4 requests are scheduled at time $t = 0$ and these passengers are assumed to accept the ride offers, which initialize the simulation episode. The following $N=8$ requests are scheduled one by one and arrive 4 minutes apart over an interval of 28 minutes. Dynamic routing is conducted using the AltMin algorithm developed in \cite{guan2019dynamicrouting}. All numerical values assumed are synthetic, and do not correspond to any actual travel data.

The first set of parameters we need to choose are the bounds in (\ref{action_range}). Suppose that the alternative to the SMoDS is an exclusive non-ride sharing service, where the waiting time is set as $\frac{2}{3}\text{EWT}^*$, and the riding time is the direct travel time from the origin to the destination. We empirically choose $\underline {a_{t_r}} = 0.5$ and $\overline{a_{t_r}} = 0.9$ if the total travel time using the SMoDS does not exceed 1.5 times that with the alternative. Otherwise, we choose $\underline {a_{t_r}} = 0.2$ and $\overline{a_{t_r}} = 0.6$. We note that the total travel time is the sum of the waiting time and riding time. In order to compute the reward, $\text{EWT}(t)$ is needed, which is determined by $F_{\text{EWT}}(\cdot)$ as in (\ref{calculate_EWT}). We approximate this function and compute $\text{EWT}(t)$ as the average of the waiting time for four passengers whose pickup locations are at the four corners of the square where all requests are assumed to originate.   
As has been discussed in Section \ref{action} and \ref{reward}, accurate derivations of $\underline {a_{t_r}}$, $\overline{a_{t_r}}$ and $\text{EWT}(t)$ require accurate knowledge of the alternative transportation options, passenger behavioral model and distribution of requests, which are beyond the scope of this manuscript. Here in this section, we adopt approximations that are sufficiently reasonable to illustrate the central idea of regulating $\text{EWT}(t)$ via $\pi^*(\cdot)$ by leveraging passenger empowerment. 

\subsection{Results of Computational Experiments}
\label{results}

With the problem setup described above in Section \ref{setup}, we apply the Exact-DP algorithm outlined in Algorithm \ref{DP_exact} and derive $\pi^*(\cdot)$ under various $\text{EWT}^*$ values. The results are summarized in Fig. \ref{fig:basic}. 
The subplots in the top row correspond to $\text{EWT}^* = 4$, 5, and 6 minutes, respectively, indicated by the black dashed lines, and show how the regulation curves (marked in blue), which denote $\mathbb{E}[\text{EWT}(t)]$ under the optimal policy $\pi^*(\cdot)$ derived using E-DP($N$), vary with $t$. 
For comparison, baseline curves (marked in orange) are also plotted, which denote $\text{EWT}(t)$ that is obtained with the assumption that each ride offer is accepted and every passenger gets on board, which is a deterministic computation. That is, none of the passengers is empowered in this case. It is easy to see that the orange curve illustrates the $\text{EWT}(t)$ that corresponds to the top path shown in Fig. \ref{fig:tree}, while the blue one corresponds to the expectation of all of the paths shown in this figure. It is therefore not surprising that the baseline curve is identical in the three subplots since it does not depend on $\text{EWT}^*$. The subplots in the bottom row of Fig. \ref{fig:tree} show the acceptance rate under $\pi^*(\cdot)$ for each passenger and the mean across all passengers for the same case of $\text{EWT}^*$ values as in the top row. 

The plots in Fig. \ref{fig:basic} clearly demonstrate that by using the optimal policy $\pi^*(\cdot)$ derived via E-DP($N$), we can effectively regulate $\text{EWT}(t)$ around $\text{EWT}^*$. The best choice of $\text{EWT}^*$, however, is not addressed in this manuscript and beyond its scope. The second observation from Fig. \ref{fig:basic} is that the expectation of the desired probability of acceptance of the passengers increases with $\text{EWT}^*$. This could be due to the fact that $\text{EWT}(t)$ is a metric of the balance between demand and supply within the SMoDS platform. That is, a higher $\text{EWT}(t)$ implies that demand exceeds supply and therefore, the agent desires to accommodate more passengers on board and hence the desired probability of acceptance increases.

It should be noted that the focus here has been to design a specific SMoDS-centric variable, which is the optimal desired probability of acceptance, in the form of an optimal policy of the MDP, which has been carried out under the assumption that the actual probability of acceptance by the empowered passenger will equal this optimal action. To ensure that this assumption is valid, a suitable dynamic tariff needs to be designed; with an increasing $\text{EWT}(t)$, to maintain the corresponding higher acceptance rate, this dynamic tariff should perhaps be suitably decreased. 
Therefore, the overall impacts of increased $\text{EWT}^*$ on the total revenue depend on two competing contributions, one is increased ridership and the other is decreased revenue per ride. Hence $\text{EWT}^*$ should be suitably determined to tune this trade-off in order to achieve desired performances of the SMoDS platform. 
These points are not discussed in detail in this manuscript, and the readers are referred to \cite{guan2019cpt}.    
\begin{figure*}[h!]
	\centering
	\includegraphics[width = 0.95\textwidth]{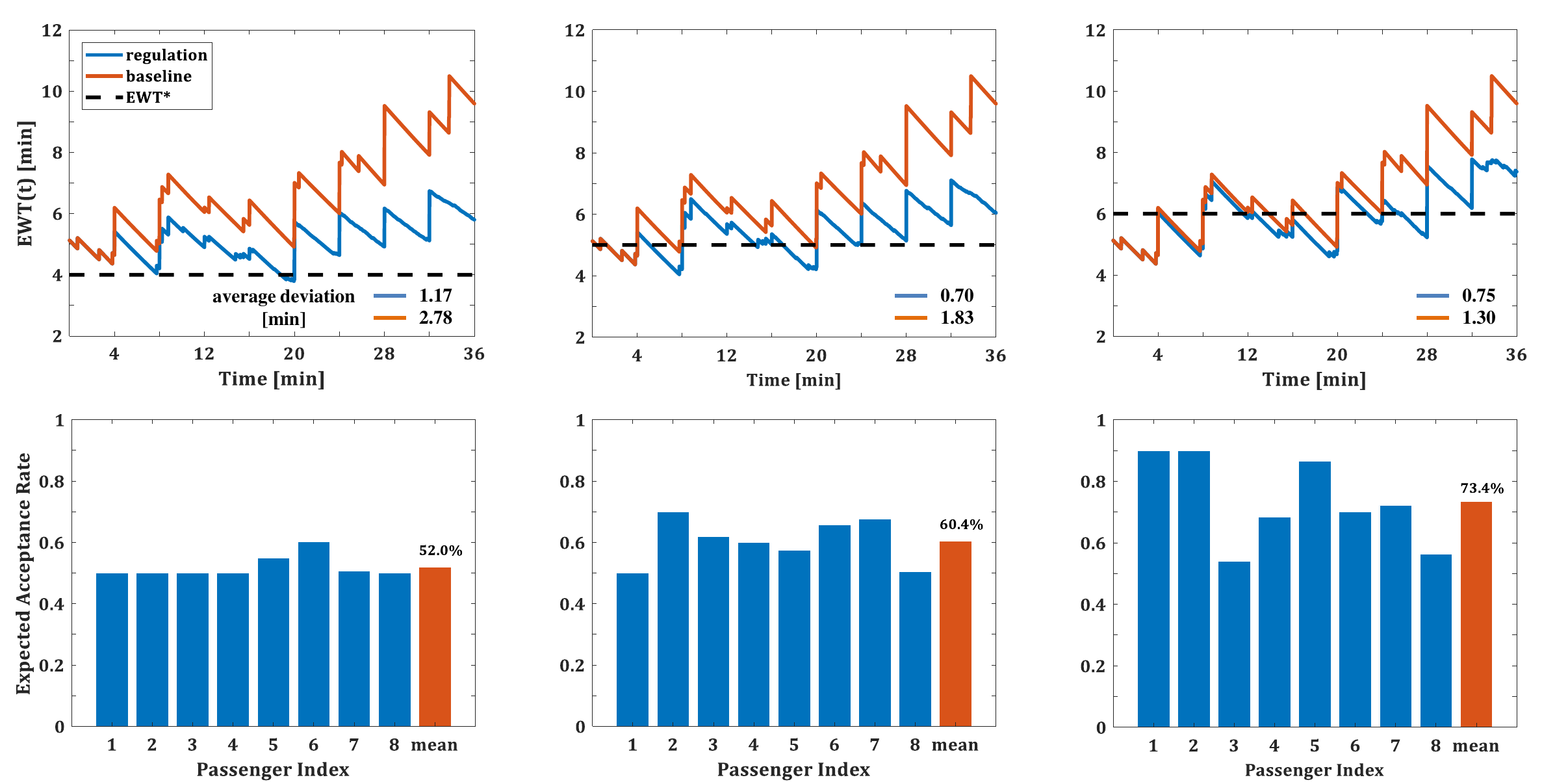}
	\caption{Regulation of $\text{EWT}(t)$ around $\text{EWT}^*$ and the incurred expected acceptance rates of each passenger for different $\text{EWT}^*$ values.}
	\label{fig:basic}
\end{figure*}

We also evaluated E-DP($N$) with a time-varying $\text{EWT}^*$, and the results are shown in Fig. \ref{fig:jump}. That is, $\text{EWT}^*$ = 4min for the first 20 minutes and switched to 6min at $t=20$. It can be seen that E-DP($N$) helps achieve a tight regulation in this case as well. These results support our argument above that $\text{EWT}^*$ could be actively tuned to adapt to the actual conditions of the system during practical operations.  
\begin{figure}[h!]
	\centering
	\includegraphics[width=0.6\textwidth]{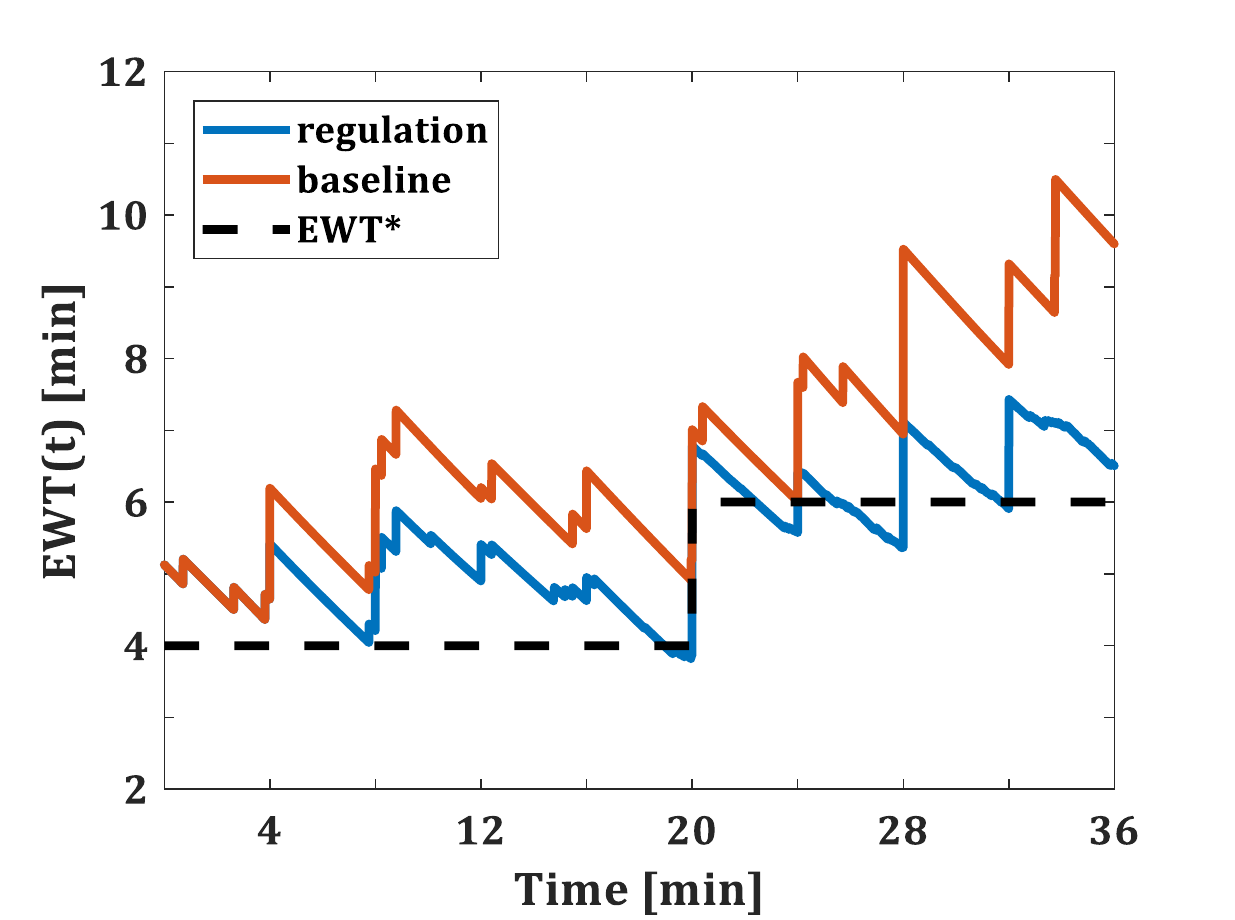}
	\caption{Regulation of $\text{EWT}(t)$ around a time-varying $\text{EWT}^*$.}
	\label{fig:jump}
\end{figure}

We now proceed to the evaluation of H-DP($\tilde{N}$), which is heuristic, and the number of lookahead steps extends up to $\tilde{N} \leq N$ rather than all the way to $N$. The resulting performance of the algorithm for the case where $\text{EWT}^* = 5$ minutes is illustrated in Fig. \ref{fig:heuristics}, where $\tilde{N}$ is varied from 0\footnote{When $\tilde{N} = 0$, actions are determined by directly comparing $F_{\text{EWT}}(s_{t_r^+})$ values, instead of using rewards and value functions in lines \ref{algo_compare_start} through \ref{algo_compare_end} in Algorithm \ref{DP_exact} for cases when $\tilde{N} > 0$.} to 8. This performance is quantified in the form of the average deviation of $\text{EWT}(t)$ from $\text{EWT}^*$, i.e., the negative of the average rewards defined in (\ref{reward_function}). Noting that the performance of H-DP(1) is fairly close to that with H-DP(8)/E-DP(8) in Fig. \ref{fig:heuristics}. It is clear that H-DP($\tilde{N}$) is a useful tool for application in general cases and can lead to a close-to-optimal solution with computational ease. 
\begin{figure}[h!]
	\centering
	\includegraphics[width=0.6\textwidth]{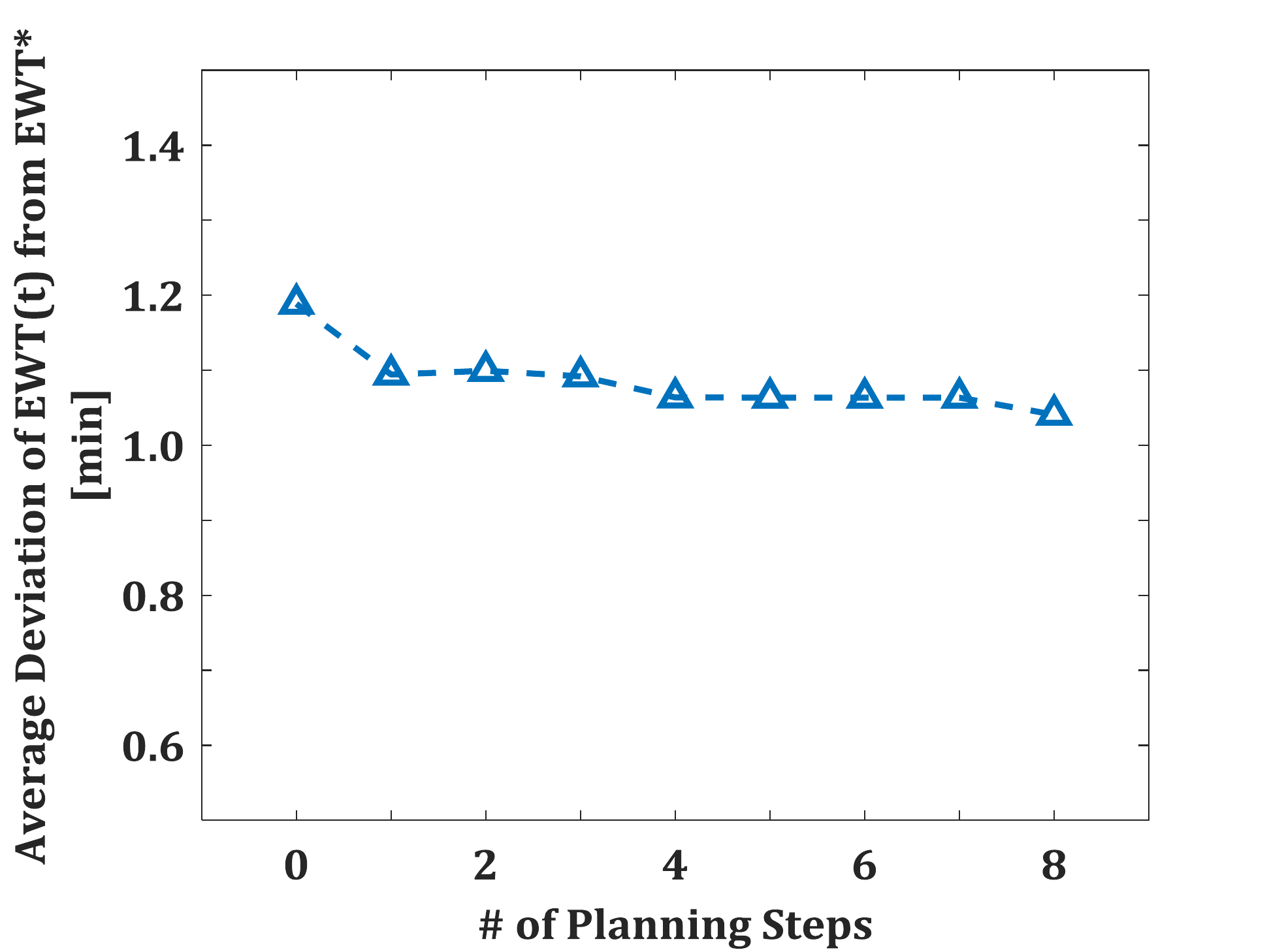}
	\caption{Average deviation of $\text{EWT}(t)$ from $\text{EWT}^*$ with respect to the number of lookahead steps $\tilde{N}$.}
	\label{fig:heuristics}
\end{figure}

\subsection{Remarks}
\label{remarks}

When the state space gets more complicated, either due to fairly large number of requests and vehicles, or the adoption of an online setup, one can exploit function approximation that has been briefly discussed in Section \ref{value} and develop a broad range of planning and learning algorithms. Notably, H-DP($\tilde{N}$) would be quite useful in these complex cases. 

\section{Concluding Remarks}
\label{conclusions}

In this manuscript, we propose a discrete time Markov Decision Process (MDP) formulation and develop a Dynamic Programming (DP) algorithm that solves the MDP towards dynamic pricing for Shared Mobility on Demand Service (SMoDS). 
The proposed MDP formulation is a versatile framework explicitly accommodating passenger behavioral modeling that leads to desired system level performances, and can be readily extended to a broad range of scenarios. 
The DP algorithm derives the optimal policy that regulates the Estimated Waiting Time $\text{EWT}(t)$ around a target value $\text{EWT}^*$ and therefore enhances the balance between demand and supply within the SMoDS platform, for a specific scenario of the MDP with an offline setup. 
Computational experiments are carried out that demonstrate effective regulation of $\text{EWT}(t)$ around $\text{EWT}^*$ as a proof of concept that performance metrics can be improved significantly via leveraging passenger empowerment, for various $\text{EWT}^*$ values and for a time-varying $\text{EWT}^*$. 
The heuristic version of the DP algorithm, H-DP($\tilde{N}$), could be exploited as the lookahead search algorithm when extended to large state spaces or online setups. 
The MDP formulation and the DP algorithm, together with our previous works on the AltMin dynamic routing algorithm in \cite{guan2019dynamicrouting} and Cumulative Prospect Theory (CPT) based passenger behavioral modeling in \cite{guan2019cpt}, provide a complete solution to the SMoDS design.     

Future works include developing integrated learning and planning algorithms for MDPs with a large state space or an online setup, and investigating disciplines that guide the choice of appropriate $\text{EWT}^*$ values leading to the desired combination of revenue and ridership for the SMoDS platform. The integration of the CPT based passenger behavioral modeling with the MDP and hence directly designing dynamic tariffs is of interest as well.

\section*{Acknowledgments}

This work was supported by the Ford-MIT Alliance. 

\section*{Appendix: Proofs of Theorems}
\label{appendix}

\vspace{2mm}
\begin{proof}[Proof of Theorem \ref{thm:optimal_action}] 
	\label{proof:optimal_action}
	Expand (\ref{value_state_rewite}), we have
	\begin{equation}
		\label{backup_value}
			\begin{split}
			v^*( s_{t_r}) & = \max_{a_{t_r} \in \Big [\underline{a_{t_r}}, \overline{a_{t_r}} \Big ]}  \bigg\{ \mathcal{R}_{s_{t_r}}^{a_{t_r}} + \gamma  \mathbb{E} \Big[ v^*(s_{t_r + \Delta t_r}) \; \Big| \; s_{t_r} \Big] \bigg\} \\
			& = \max_{a_{t_r} \in \Big [\underline{a_{t_r}}, \overline{a_{t_r}} \Big ]} \bigg\{  a_{t_r} \Big\{r_{t_r}^{[1]} + \gamma \mathbb{E} \Big[ v^* \Big( s_{t_r+\Delta t_r} ^ {[1]} \Big) \Big ] \Big\} + (1 - a_{t_r}) \Big\{r_{t_r}^{[0]} + \gamma \mathbb{E} \Big[ v^* \Big( s_{t_r+\Delta t_r}^{[0]}\Big) \Big ] \Big\} \bigg\} \\
			& = \max_{a_{t_r} \in \Big \{ \underline{a_{t_r}}, \overline{a_{t_r}} \Big \}} \bigg\{ a_{t_r} \Big\{r_{t_r}^{[1]} + \gamma \mathbb{E} \Big[ v^* \Big( s_{t_r+\Delta t_r} ^ {[1]} \Big) \Big ] \Big\} + (1 - a_{t_r}) \Big\{r_{t_r}^{[0]} + \gamma \mathbb{E} \Big[ v^* \Big( s_{t_r+\Delta t_r}^{[0]}\Big) \Big ] \Big\} \bigg\}
		\end{split}
	\end{equation}
	The first equality is equivalent to Bellman Optimality Equation and holds by definition. The second equality holds due to expanding $\mathcal{R}_{s_{t_r}}^{a_{t_r}}$ by taking the expectation of $D_{t_r}$. The third equality holds because the right-hand side of the second equality is linear in $a_{t_r}$, since the remaining expectations are taken of the distribution of $\Omega_{t_r + \Delta t_r}$ and hence $r_{t_r}^{[1]}$, $v^* \Big( s_{t_r+\Delta t_r} ^ {[1]} \Big)$, $r_{t_r}^{[0]}$, or $v^* \Big( s_{t_r+\Delta t_r}^{[0]} \Big)$ do not depend on $a_{t_r}$, therefore the maximum must be reached at either end point. Here $r_{t_r}^{[1]} = - \frac{1}{\Delta t_r} \int_{0^+}^{{(\Delta t_r)}^-}  |F_{\text{EWT}} [F_{\text{ID}}(F_{\text{DR}}(s_{t_r}), \tau)] - \text{EWT}^*| \; d\tau$, $s_{t_r+\Delta t_r} ^ {[1]} = [ F_{\text{ID}} ( F_{\text{DR}}(s_{t_r}), \Delta t_r ) , \omega_{t_r+ \Delta t_r} ]$, $r_{t_r}^{[0]} = - \frac{1}{\Delta t_r} \int_{0^+}^{{(\Delta t_r)}^-}   |F_{\text{EWT}}[F_{\text{ID}}(s_{t_r^-}, \tau)] - \text{EWT}^*| \; d\tau$, and $s_{t_r+\Delta t_r}^{[0]} = [ F_{\text{ID}} ( s_{t_r^-} , \Delta t_r ), \omega_{t_r + \Delta t_r} ]$\footnote{We omit the decisions from previous passengers in the superscripts as the formulas of $r_{t_r}^{[1]}$, $s_{t_r+\Delta t_r} ^ {[1]}$, $r_{t_r}^{[0]}$ and $s_{t_r+\Delta t_r} ^ {[0]}$ hold for any scenario.}. 
\end{proof}

\vspace{2mm}
\begin{proof}[Proof of Theorem \ref{thm:markov}] 
	\label{proof:markov}
	According to (\ref{transition_combine}), we have 
	\begin{equation}
		\label{markov_property}
		\mathbb{P} \Big [S_{k + 1} = s_{k + 1} \; \Big| \; S_k = s_k, \cdots, S_1 = s_1 \Big ] = \mathbb{P} \Big [ S_{k+1} = s_{k+1} \; \Big| \;  S_k = s_k \Big ] 
	\end{equation} 
	$\forall k \in \mathbb{Z}_{> 0}$ and $ \{ s_1, \cdots, s_{k+1} \} \subset \mathcal{S}$. Obviously, (\ref{markov_property}) holds for the first and third subcases in (\ref{transition_combine}). (\ref{markov_property}) holds for the second subcase in (\ref{transition_combine}) because $\mathbb{P} \Big [ S_{k+1} = s_{k+1}\; \Big | \; S_k = s_k \Big ] = \mathbb{P} \Big [ S_{k+1} = s_{k+1}\; \Big | \; S_k = s_k, S_{k^-} = s_{k^-} \Big ] $. Hence $\forall k \in \mathbb{Z}_{>0}$, $S_k$ are Markov.
\end{proof}

%





	
\bibliographystyle{plain}

\bibliography{references}


\end{document}